\documentclass[a4paper,11pt]{article}

\usepackage{amsmath, amsthm, color, amssymb,  amsfonts, amscd, epsfig}

\newtheorem{theo}{Theorem}[section]
\newtheorem{defi}[theo]{Definition}
\newtheorem{lem}[theo]{Lemma}
\newtheorem{prop}[theo]{Proposition}
\newtheorem{cor}[theo]{Corollary}
\newtheorem{rem}[theo]{Remark}

\newcommand{\R}{\mathbb R}

\newcommand{\C}{\mathbb C}
\newcommand{\Heis}{\mathbb H}
\newcommand{\B}{\mathcal B}

\newcommand{\Proj}{\mathcal P}

\newcommand{\G}{\mathbb G}

\newcommand{\F}{\mathcal F}
\newcommand{\FG}{\mathcal F_{SE(2)}}
\newcommand{\Fock}{\mathbb F}
\newcommand{\PF}[2]{\Proj\Fock^{#1}_{#2}}

\newcommand{\Hil}{\mathcal H}

\newcommand{\Sw}{\mathcal S}

\newcommand{\sn}[1]{\| #1 \| {}_{\Hil_\Omega}}
\newcommand{\spr}[2]{\langle #1 , #2 \rangle_{\Hil_\Omega}}

\newcommand{\dirint}[1]{\int_{#1}^{\oplus}}

\title{Reproducing kernel Hilbert spaces of CR functions for the Euclidean Motion group}
\author{D. Barbieri\footnote{{\tt davide.barbieri8@gmail.com}. CAMS, EHESS/CNRS. 190 av. de France, 75244 Paris. Author partially supported by Progetto Strategico Giovani MIUR-Unibo.}, G. Citti\footnote{{\tt giovanna.citti@unibo.it}. Dept. of Mathematics, p.ta S. Donato 5, 40126 Bologna.}}

\begin{document}

\maketitle

\vspace{-4pt}
\begin{abstract}
We study the geometric structure of the reproducing kernel Hilbert space associated to the continuous wavelet transform generated by the irreducible representations of the Euclidean Motion $SE(2)$. A natural Hilbert norm for functions on the group is constructed that makes the wavelet transform an isometry, but since the considered representations are not square integrable the resulting Hilbert space will not coincide with $L^2(SE(2))$. The reproducing kernel Hilbert subspace generated by the wavelet transform, for the case of a minimal uncertainty mother wavelet, can be characterized in terms of the complex regularity defined by the natural $CR$ structure of the group. Relations with the Bargmann transform are presented.\vspace{3pt}\\
{\tt keywords}: Euclidean Motion group, reproducing kernel Hilbert spaces, uncertainty principle, CR functions.\\
{\tt 2010 MSC}: 46E22, 32V10, 22E45, 81R30, 43A32
\end{abstract}

\section{Introduction}

Several results in the harmonic analysis of the continuous wavelet transform rely on the reproducing property associated to square integrable unitary irreducible representations (UIR) (see \cite[Ch. 14]{Dixmier} and \cite{Fuhr}). The typical setting is the following: let $\G$ be a Lie group with Haar measure $\mu$, $\pi$ a square integrable UIR of $\G$ on the Hilbert space $\Hil$, and $\psi_0$ a nonzero vector in $\Hil$, that we will call mother wavelet. Then for any $f \in \Hil$ the identity
\begin{equation}\label{eq:reproducing0}
f = \int_{\G} \langle f, \pi(g)\psi_0\rangle_{\Hil} \, \pi(g)\psi_0 \, d\mu(g)
\end{equation}
holds in weak sense, that means that for any $f_1, f_2 \in \Hil$ the \emph{analysis operator}
$$
Af(g) \doteq \langle f, \pi(g)\psi_0\rangle_{\Hil}
$$
is an isometry from $\Hil$ to $L^2(\G)$, i.e. it satisfies $\langle f_1, f_2\rangle_{\Hil} = \langle Af_1, Af_2\rangle_{L^2(\G)}$.
\newpage

Classical examples are those associated to the $ax + b$ group and to the Heisenberg group \cite{Daubechies}, but generalizations to several groups were made \cite{Perelomov1972, AAG, Fuhr}. The abstract setting gives rise to a reproducing kernel Hilbert subspace of $L^2(\G)$, that is the target space of the analysis operator. A notable example is given by the case of the Heisenberg group when the mother wavelet is a minimum of the uncertainty principle. In this case the continuous wavelet transform, up to a weight, is the well-known Bargmann transform \cite{BargmannII}, and its target space is completely characterized by a complex regularity condition, i.e. it is a space of entire functions.

Notable applications of such a construction can be found in signal analysis \cite{Grochenig, DKMSST} and path integrals \cite{KlauderPathIntegrals}, while we recall that the issue of complex regularity is a fundamental tool for sampling and interpolation issues \cite{SW,Seip}.
However, when the group $\G$ does not admit square integrable representations, eventually modulo a non-trivial isotropy subgroup (e.g. the center \cite{CorwinGreenleaf}), the reproducing formula (\ref{eq:reproducing0}) does not formally make sense since $Af$ does not belong to $L^2(\G)$.

One of the simplest group that is not square integrable but nevertheless provides interesting applications \cite{IK, CS, Duits} is the group of  rigid motions of the Euclidean plane $SE(2) = R^2_q \rtimes S^1_\theta$, that is the noncommutative Lie group obtained as semidirect product between translation and counterclockwise rotations of the plane (see e.g. \cite{V,Sugiura}), and whose Haar measure is the Lebesgue one. Its composition law is given by $(q',\theta') \cdot (q,\theta) = (q' + r_{\theta'}q, \theta' + \theta)$ where $r_\theta$ is the ordinary counterclockwise rotation of an angle $\theta$.
Its Lie algebra can be defined in terms of left invariant vector fields, which read \cite{CS}
\begin{eqnarray}\label{generators}
X_1 = -\sin\theta\partial_{q_1} + \cos\theta\partial_{q_2} \ , \ X_2 = \partial_{\theta}\\
X_3 = [X_1,X_2] = \cos\theta\partial_{q_1} + \sin\theta\partial_{q_2}\nonumber\ .
\end{eqnarray}

This group is locally equivalent to the three dimensional Heisenberg group, which is its metric tangent cone (in the sense of Gromov, see \cite{Bellaiche}), but globally it is not nilpotent.

Its inequivalent infinite dimensional irreducible representations (see e.g. \cite{Sugiura}) are parametrized by $\R^+ \approx \R^2 / S^1$. For any $\Omega \in \R^+$ there exists an irreducible representation of $SE(2)$ on the Hilbert space $L^2(S^1)$ given by
\begin{equation}\label{eq:S1representation}
\Pi^\Omega(q,\theta) u(\varphi) = e^{-i\Omega (q_1\cos\varphi + q_2\sin\varphi)} u(\varphi - \theta)\ .
\end{equation}

Due to the phase term, these representations are not square integrable, i.e. there exist no $u, v \in L^2(S^1)$ such that the function $\langle \Pi^\Omega(q,\theta) u, v\rangle_{L^2(S^1)}$ belongs to $L^2(SE(2))$.
In order to circumvent this problem, several different strategies were proposed. In \cite{DB}, the author restricted the orbit of the representation to a two dimensional submanifold obtained as cotangent bundle of a coadjoint orbit, obtaining as a byproduct the requirement for a compactly supported admissible mother wavelet; this method could later be extended to general semidirect product groups (see e.g. \cite{AAG}). Following a different approach, in \cite{IK} the authors could re-establish square integrability by making use of reducible representations constructed as direct integrals over finite intervals of parameters, a procedure that does not require compact support for the mother wavelet. Working with reducible representations and obtaining  admissibility conditions over the fiducial vector turns out to be a rather general and consistent procedure \cite{Fuhr} and for the case of the $SE(2)$ group it was recently applied also to the left quasi-regular representation, in the framework of image processing \cite{Duits}.

We will propose a different construction, that does not rely on square integrability but rather on the characterization of the reproducing kernel Hilbert space of the continuous wavelet transform, defined by the irreducible representation on the whole group. We consider, for any normalized vector $u_0 \in L^2(S^1)$, the analysis operator acting on vectors $\Phi \in L^2(S^1)$ as
\begin{equation}\label{eq:analysis}
A^\Omega : \Phi \mapsto A^\Omega\Phi(q,\theta) \doteq \langle \Pi^\Omega(q,\theta) u_0, \Phi\rangle_{L^2(S^1)} .
\end{equation}
We will first show that $A^\Omega$ is an isometry on a Hilbert space of distributions, that we will denote with $\Hil_\Omega(SE(2))$, which is a direct summand of $L^2(SE(2))$ and provides an explicit expression for the reproducing kernel Hilbert norm. The construction of such a Hilbert space, together with the proof of its main properties, will be worked out in Section \ref{sec:reproducing}.
Subsequently, we will choose as mother wavelet a minimizer of the generalized uncertainty principle associated to the group $SE(2)$ (see \cite{FS}), given by
\begin{equation}\label{eq:minmu}
u^\mu(\varphi) = c_\mu e^{\mu \cos\varphi}
\end{equation}
where $\mu$ is a nonzero positive parameter and $c_\mu$ is a normalization constant.  
In this case we will show in Section \ref{sec:CR} that the reproducing kernel Hilbert space naturally associated to (\ref{eq:analysis}) is characterized by a complex regularity condition which generalizes the analiticity condition to odd dimensional manifold. This notion of complex regularity is called $CR$ condition (see e.g. \cite{BER}), and for any $\lambda \in \R$ a function $F : \R^2 \times S^1 \to \C$ will be said to be $CR^\lambda$ if it satisfies
$$
(X_2 + i \lambda X_1) F = 0
$$
where $X_1$ and $X_2$ are the first order differential operators (\ref{generators}).
Our main theorem will then be the following.
\begin{theo}\label{theo:main}
Let $\Omega \in \R^+$, $\mu \in \R \setminus \{0\}$ and set $\lambda = \frac{\mu}{\Omega}$. Then
\begin{itemize}
\item[i)] for all unit vector $u_0 \in L^2(S^1)$, the map $A^\Omega$ given by (\ref{eq:analysis}) is an isometry of $L^2(S^1)$ on $\Hil_\Omega(SE(2))$;
\item[ii)] for $u^\mu$ given by (\ref{eq:minmu}), $A^\Omega$ maps $L^2(S^1)$ onto $\Hil_\Omega(SE(2)) \cap CR^\lambda$.
\end{itemize}
\end{theo}
\noindent The proof of $i)$ will be given in Section \ref{sec:reproducing}, while the proof of $ii)$ will be given in Section \ref{sec:CR}.
\newpage

This $CR$ regularity is directly related to the contact geometry of the group, defined by the nonintegrable distribution of tangent hyperplanes generated by the vector fields $X_1$ and $X_2$. Consider the manifold $\mathcal{M} = \R^2 \times S^1$ as a real submanifold of $\R^4 \approx \C^2$ endowed with the same complex structure used for the ordinary Bargmann transform:
\begin{equation}\label{eq:realsubm}
\mathcal{M} = \{(z_1,z_2) \in \C^2, z_j = (p_j + i q_j)\, j=1,2 \, , \ \textrm{such that} \ p_1^2 + p_2^2 = 1\}\ .
\end{equation}
The complex structure of $\C^2$ induces an almost complex structure on the contact planes, with respect to which the vector $X_2 + i X_1$ is antiholomorphic. For this reason we will call the analysis $A^\Omega$ with respect to a mother wavelet as in (\ref{eq:minmu}) the $SE(2)$-Bargmann transform. With Theorem \ref{theo:Bargmann} in Section \ref{sec:CR}, we will show that such a transform can be obtained as a restriction of the ordinary Bargmann transform to circles in the cotangent variables. We also recall that Fock-Bargmann spaces in $\C^n$, of poly-analytic functions, constitute the subject of recent research such as \cite{Abreu, AbreuGroch}.

An application of the theory presented, that actually motivated the whole construction, regards the behaviour of brain visual cortex when subject to oriented stimuli. A geometric model that is capable to reproduce experimental results of cortical activities for a classical experimental setting was given in \cite{BCSS}, and its main features can be properly interpreted in terms of the $SE(2)$-Bargmann transform.

\noindent {\bf Acknowledgements} Much of the motivations and hints for the present work arose in discussions with professor A. Sarti, that the authors would like to acknowledge as an important source of inspiration. Also the authors would like to thank professors F. Ricci and E. De Vito for helpful discussions and suggestions.

\section{Hilbert spaces and measure decomposition}\label{sec:Hilbert}

In this section we will define the Hilbert spaces $\Hil^\Omega$ and $\Hil_\Omega$ that will allow to define the Hilbert space $\Hil_\Omega(SE(2))$ where the continuous wavelet transform maps $L^2(S^1)$ isometrically. In short, $\Hil^\Omega$ is a space of distributions on the plane that is isomorphic to $L^2(S^1)$, while $\Hil_\Omega$ is its pullback through the Fourier transform.

We will choose as (unitary) Fourier transform $\F :L^2(\R^2) \to L^2(\hat{\R}^2)$
\begin{displaymath}
\F f (k) = \hat{f}(k) = \frac{1}{2\pi}\int_{\R^2} e^{-ik \cdot x} f(x) dx
\end{displaymath}
while we will use the right-antihermitian convention for both $L^2$ scalar products and tempered distributions $\Sw'$ on Schwartz class functions $\Sw$.

In order to motivate the definitions of $\Hil^\Omega$ and $\Hil_\Omega$, let us consider the left quasi-regular representation of $SE(2)$ on $f \in L^2(\R^2)$:
\begin{equation}\label{eq:leftregularrep}
\left(L(g)f\right)(x) = f(g^{-1}\,x)
\end{equation}
where $g=(q,\theta) \in SE(2)$ and $g^{-1}\,x = r_{-\theta}(x - q)$.
By Fourier transform we obtain a representation on $L^2(\hat\R^2)$ unitarily equivalent to (\ref{eq:leftregularrep}):
\begin{equation}\label{eq:fourierlrr}
\big(\hat L(g) \hat f\big)(k) = e^{-ik\cdot q} \hat f(r_{-\theta}k)
\end{equation}
where $\hat L(g) \doteq \F L(g) \F^{-1}$, indeed
\begin{eqnarray*}
\left(\F L(g) f\right)(k) & = & \frac{1}{2\pi}\int e^{-ik\cdot x}f(g^{-1}\,x)dx = \frac{1}{2\pi}\int e^{-ik\cdot g\,x}f(x)dx\\
& = & e^{-ik\cdot q} \left(\F f\right)(r_{-\theta}k) .
\end{eqnarray*}

By (\ref{eq:fourierlrr}) we see that $\hat L(g)$ acts in an invariant way on each circle of the domain of $\hat{f}$, so it can be reduced by considering its action on functions restricted to a circle of the Fourier domain. This is generally not allowed for elements of $L^2(\hat\R^2)$, and should be properly intended in terms of the direct integral decomposition associated to the usual polar coordinates
\begin{equation}\label{eq:dirint}
L^2(\hat{\R}^2) \approx \dirint{\R^+} \Hil^\Omega \,\Omega d\Omega
\end{equation}
where each $\Hil^\Omega$, that will be rigorously defined in a while, is isomorphic to $L^2(S^1)$, and $\Omega \in \R^+$ stands for the radius of the corresponding circle.

This approach to the decomposition of the quasi-regular representation of the $SE(2)$ group into irreducible ones was explicitly discussed already in \cite{V}, and provides a very special case of measure decomposition (see e.g. \cite[Prop. 3.29]{Fuhr}). It is particularly useful in this case since it permits to emphasize that one should be rather careful when coming back from Fourier to spatial variables, since the localization (on circles) in the Fourier domain implies a delocalization in the spatial domain. This passage indeed leads to the loss of square integrability, and we will see in next section that this is the same phenomenon that happens to the continuous wavelet transform. However, since $\hat L$ is reducible on $L^2(\hat\R^2)$, so it is $L$ on $L^2(\R^2)$ and we will denote the irreducible Hilbert spaces in the spatial domain with $\Hil_\Omega$.

\subsection{The Hilbert spaces $\Hil^\Omega$ and $\Hil_\Omega$}\label{sec:HilOmega}

For $\Omega \in \R^+$ and $u \in L^2(S^1)$, define the functional $\hat{T}^\Omega_{u}$ on $\Sw(\hat\R^2)$ by
$$
\langle \hat{T}^\Omega_u , \psi \rangle_{\Sw'\,\Sw} \doteq \int_0^{2\pi} u(\varphi) \overline{\psi(\Omega\cos\varphi, \Omega\sin\varphi)}d\varphi \quad \forall \ \psi \in \Sw(\hat \R^2) .
$$
$\hat{T}^\Omega_{u}$ is a continuous (anti) linear functional on $\Sw(\hat\R^2)$ since by the Cauchy-Schwartz inequality
\begin{align*}
|\langle \hat{T}^\Omega_u , \psi \rangle_{\Sw'\,\Sw}| & \leq \|u\|_{L^2(S^1)} \left(\int_0^{2\pi} |\psi(\Omega\cos\varphi, \Omega\sin\varphi)|^2d\varphi\right)^\frac12\\
& \leq (2\pi)^\frac12 \|u\|_{L^2(S^1)} \sup |\psi|
\end{align*}
which implies continuity (see e.g. \cite{Rudin}). Moreover, the map $u \mapsto \hat{T}^\Omega_{u}$ is injective since if $v \in L^2(S^1)$
$$
\int_0^{2\pi} v(\varphi) \phi(\varphi) d\phi = 0 \ \ \forall \phi \in \mathcal{C}_c(S^1) \quad \Rightarrow \quad v = 0 \ \ \textnormal{a.e. in} \ S^1.
$$

\begin{defi}\label{def:hilspaces}
For any $\Omega \in \R^+$, define the Hilbert spaces of distributions
\begin{align*}
\Hil^\Omega & \doteq \{\hat{T}^\Omega_{u} : u \in L^2(S^1)\} \subset \Sw'(\hat\R^2)\\
\Hil_\Omega & \doteq \F^{-1}\Hil^\Omega \subset \Sw'(\R^2)
\end{align*}
endowed with the scalar product
$$
\langle \F^{-1} \hat{T}^\Omega_{u}, \F^{-1} \hat{T}^\Omega_{v}\rangle_{\Hil_\Omega} = \langle \hat{T}^\Omega_{u}, \hat{T}^\Omega_{v}\rangle_{\Hil^\Omega} \doteq \langle u,v\rangle_{L^2(S^1)} .
$$
\end{defi}

%

The space $\Hil^\Omega$ results by considering functions on the circle as restrictions of functions on the plane.
Given a function $\hat{f} \in \Sw(\hat \R^2)$, for any fixed $\Omega \in \R^+$ denote with $\hat{f}^\Omega \in \Sw(S^1)$ its restriction on the circle of radius $\Omega$ as
\begin{equation}\label{eq:correspondence}
\quad\hat{f}^\Omega(\varphi) = \hat{f}(\Omega\cos\varphi, \Omega\sin\varphi) .
\end{equation}
Since $\hat{f}^\Omega \in L^2(S^1)$, define the linear operator $\Proj^\Omega: \Sw(\hat\R^2)\rightarrow \Hil^\Omega$
$$
\Proj^\Omega \hat{f} \doteq \hat{T}^\Omega_{\hat{f}^\Omega } \ .
$$
The space $\Hil^\Omega$ corresponds then to the closure of $\Proj^\Omega(\Sw(\hat\R^2))$ in the $\Hil^\Omega$ norm. Moreover, by the usual polar coordinates change of variables
\begin{equation}\label{eq:Fubini}
\|\hat f\|_{L^2(\hat\R^2)}^2 = \int_{\R^+} \int_{S^1} |\hat{f}^\Omega(\varphi)|^2 \Omega d\Omega
\end{equation}
so that, by Fubini theorem, (\ref{eq:correspondence}) provides an $L^2(S^1)$ function for a.e. $\Omega \in \R^+$ whenever $\hat{f} \in L^2(\hat\R^2)$. 
This gives the integral decomposition (\ref{eq:dirint}) and, since $\hat{f}^\Omega$ allows to define a tempered distribution, the operator $\Proj^\Omega$ can be extended to $L^2(\hat\R^2)$ for a.e. $\Omega \in \R^+$.
We observe also that $\Proj^\Omega$ can be expressed in terms of the group Fourier transform. In order to do this, recall the following \cite[Chapt. 4, Prop. 3.4]{Sugiura}.
\begin{lem}[Group Fourier transform]
Let $f \in \Sw(\R^2 \times S^1)$ and $\Omega \in \R^+$. Then the $SE(2)$ Fourier transform of $f$
\begin{displaymath}
\FG f (\Omega) = \int_{\R^2 \times S^1} dq d\theta f(q,\theta) \Pi^\Omega(q,\theta)
\end{displaymath}
is the (compact) integral operator on $L^2(S^1)$
$$
\FG f (\Omega) u(\varphi) = \int_{S^1} d\theta \hat{f}^{\Omega}(\varphi,\varphi - \theta) u(\theta)
$$
where we have used notation (\ref{eq:correspondence}), and the Fourier transform is performed with respect to spatial variables.
\end{lem}
Since $S^1$ has finite measure, we can perform the group Fourier transform on functions $f \in \Sw(\R^2)$, which reads
$$
\FG f (\Omega) u(\varphi) = \hat{f}^{\Omega}(\varphi) \int_{S^1} u(\theta) d\theta
$$
so that the operator $\Proj^\Omega$ can be obtained from
$$
\langle \Proj^\Omega \hat{f} , \psi \rangle_{\Sw'\,\Sw} = \frac{1}{\int_{S^1} u(\theta) d\theta} \langle \FG f (\Omega) u , \psi\rangle_{L^2(S^1)}
$$
which holds for any $u \in L^2(S^1)$.

The Hilbert space $\Hil_\Omega$ can be characterized in the following way.
\begin{prop}\label{prop:PROJ}
For $\Omega \in \R^+$ let us denote with $\Proj_\Omega : \Sw(\R^2) \rightarrow \Hil_\Omega$ the linear operator $\Proj_\Omega \doteq \F^{-1} \Proj^\Omega \F$. Then $\Proj_\Omega(\Sw(\R^2))$ is dense in $\Hil_\Omega$ and for any $f \in \Sw(\R^2)$ it holds
\begin{align*}
i) & \quad \Proj_\Omega f(x) = \frac{1}{(2\pi)^2}f \ast j_0(\Omega |\cdot|)(x) = \int_{\R^2} f(y) \, \frac{j_0(\Omega |x - y|)}{(2\pi)^2}dy\\
ii) & \quad \sn{\Proj_\Omega f} = \left(\int_{\R^2} dx \int_{\R^2}dy \, f(x) \overline{f(y)} \, \frac{j_0(\Omega |x - y|)}{(2\pi)^2}\right)^{\frac12}
\end{align*}
where $j_0$ is the order zero modified Bessel function of the first kind \cite{V}
$$
j_0(s) = \int_0^{2\pi} e^{i s \cos\varphi} d\varphi\ .
$$
Moreover, for every $f \in L^2(\R^2)$ point $i)$ provides a tempered distribution and point $ii)$ is finite for a.e. $\Omega \in \R^+$.
\end{prop}
\begin{proof}
The density of $\Proj_\Omega(\Sw(\R^2))$ in $\Hil_\Omega$ is equivalent to the density of $\Proj^\Omega(\Sw(\hat\R^2))$ in $\Hil^\Omega$.
Also, by the definition of $\Proj_\Omega$, point $i)$ reads equivalently
$$
\Proj^\Omega \hat f= \frac{1}{(2\pi)^2} \F \left(f \ast j_0(\Omega |\cdot|)\right) =\frac{1}{2\pi} \hat f(k) \left(\F j_0(\Omega |\cdot|)\right)\
$$
and in order to prove this we only need to show that $\F j_0(\Omega |\cdot|) = 2\pi\hat{T}^\Omega_1$.
This is true since
\begin{eqnarray*}
\langle j_0(\Omega |\cdot|),\psi \rangle_{\Sw'\Sw} & = & \int_{\R^2}dx \int_{0}^{2\pi} d\varphi \, e^{i\Omega|x|\cos\varphi}\overline{\psi(x)}\\
& = & 2\pi \int_{0}^{2\pi} d\varphi \, \overline{\F \psi (\Omega\cos\varphi,\Omega\sin\varphi)} = 2\pi \langle \hat{T}^\Omega_1,\F \psi \rangle_{\Sw'\Sw}\ .
\end{eqnarray*}
In order to prove point $ii)$, note that the $\Hil_\Omega$ norm of $\Proj_\Omega f$ coincides with the $\Hil^\Omega$ norm of $\hat{T}^\Omega_{\hat{f}^\Omega}$, which reads
$$
\|\hat{T}^\Omega_{\hat{f}^\Omega}\|^2_{\Hil^\Omega} = \|\hat{f}^\Omega \|^2_{L^2(S^1)} = \langle\Proj^\Omega \hat{f} , \hat{f}\rangle_{\Sw'(\hat\R^2),\Sw(\hat\R^2)} = \langle \Proj_\Omega f , f\rangle_{\Sw'(\R^2)\,\Sw(\R^2)}
$$
so that $ii)$ is a direct consequence of $i)$ and of Definition \ref{def:hilspaces}. The well posedness of $i)$ and $ii)$ for a.e. $\Omega \in \R^2$ when $f \in L^2(\R^2)$ are equivalent to the corresponding statements in the Fourier domain.
\end{proof}

\begin{rem}
The space $\Hil_\Omega$ can be equivalently obtained as the Gelfand-Raikov Hilbert space over $L^1(\R^2)$ associated to the function of positive type $\phi(x) = \frac{j_0(\Omega |x - y|)}{(2\pi)^2}$ (see e.g. \cite[\S 3.3]{Folland95}). Indeed, for each $f \in L^1(\R^2)$ point $ii)$ of Proposition \ref{prop:PROJ} is finite, and it provides a Hilbert seminorm on $L^1(\R^2)$. If we call $N_\Omega$ its null space over $L^1(\R^2)$, we obtain $\Hil_\Omega = \overline{L^1(\R^2)/N_\Omega}^{\Hil_\Omega}$.
\end{rem}




%

The direct integral decomposition (\ref{eq:dirint}) can be pulled back by Fourier transform, so that $L^2(\R^2)$ decomposes as
\begin{equation}\label{eq:dirint2}
L^2(\R^2) \approx \dirint{\R^+} \Hil_\Omega\,  \Omega d\Omega
\end{equation}
where each $\Hil_\Omega$ is invariant under the action of the left quasi-regular representation (see also \cite[\S 5]{V}). More precisely
\begin{prop}\label{prop:dirint}
Let $f \in L^2(\R^2)$. Then there exists a measurable family $\{f_\Omega\}_{\Omega \in \R^+}$, unique up to zero measure sets, such that each $f_\Omega$ belongs to $\Hil_\Omega$ and which realizes
\begin{displaymath}
\|f\|^2_{L^2(\R^2)} = \displaystyle{\int_0^\infty \|f_\Omega\|^2_{\Hil_\Omega} \, \Omega d\Omega} .
\end{displaymath}
Moreover, the following representation formula holds
\begin{displaymath}
f = \displaystyle{\int_0^\infty f_\Omega \, \Omega d\Omega}
\end{displaymath}
\end{prop}

\begin{proof}
The first identity is obtained by Plancherel theorem and (\ref{eq:Fubini}):
\begin{displaymath}
\|f\|^2_{L^2(\R^2)} = \int_{\hat\R^2} |\hat{f}(k)|^2 dk = \int_{\R^+} \|\hat{f}^\Omega\|^2_{L^2(S^1)} \, \Omega d\Omega = \int_{\R^+} \|\Proj_\Omega f\|_{\Hil_\Omega}^2 \, \Omega d\Omega .
\end{displaymath}
In order to prove the representation formula, let us first consider $f \in \Sw(\R^2)$. Then
\begin{eqnarray*}
f(x) & = & \int_{\R^2} d k \hat f(k) e^{ikx}   =
\int_{\R^2} d k \int_{\R^2} dy f(y)   e^{ik(x-y)} \\
& = & \int_0^\infty d\Omega\, \Omega\, \int_{\R^2} dy\, f(y) \int_0^{2\pi}d\varphi\, e^{i\Omega |x-y| \cos\varphi}\\
& = & \int_0^\infty d\Omega\, \Omega \int_{\R^2} dy f(y) j_0(\Omega|x-y|)= \int_0^\infty  \Proj_\Omega f(x) \, \Omega d\Omega.
\end{eqnarray*}
so the same holds for $f \in L^2(\R^2)$ provided the integral is intended in weak sense.
\end{proof}

\subsection{Isometry and reproducing formulas}\label{sec:reproducing}

The continuous wavelet transform implemented by the analysis (\ref{eq:analysis}) turns a function $\Phi$ of one $S^1$ variable into the function $A^\Omega \Phi$ of three $R^2 \times S^1$ variables, that belong to $L^2(S^1)$ in the angular variable but it is not square integrable with respect to spatial variables, but actually belongs to $\Hil_\Omega$. In this section we prove point $i)$ of Theorem \ref{theo:main} making use of this argument, and see how the distributional approach allows to obtain a weak reconstruction formula.

For $\Omega \in \R^+$ we will denote with $\Hil_\Omega(SE(2))$ the Hilbert space $L^2(S^1,\Hil_\Omega)$
endowed with the scalar product
\begin{displaymath}
\langle F, G\rangle_{\Hil_\Omega(SE(2))} = \int_{S^1} d\theta \spr{F(\cdot,\theta)}{G(\cdot,\theta)} .
\end{displaymath}
Note that, since each $\Hil_\Omega$ is invariant under the quasi-regular representation, then each $\Hil_\Omega(SE(2))$ is invariant under the left regular representation of $SE(2)$. Moreover, the natural isomorphism $L^2(\R^2 \times S^1) \approx L^2(S^1,L^2(\R^2))$ together with (\ref{eq:dirint2}) provides the direct integral decomposition
\begin{equation}\label{eq:dirint3}
L^2(SE(2)) \approx \int_{\R^+}^\oplus \Hil_\Omega(SE(2)) \Omega d\Omega
\end{equation}
so that the left regular representation decomposes accordingly. Note also that this is not a decomposition into irreducibles, which is given e.g. by \cite[Chapt. 4, Theorem 4.3]{Sugiura}. Note also that, by Proposition \ref{prop:PROJ}, for all $\Psi_1, \Psi_2 \in \Sw(SE(2))$ it holds
\begin{equation}\label{eq:spr}
\langle \Psi_1, \Psi_2 \rangle_{\Hil_\Omega(SE(2))} = \int_{S^1} \!\!\!d\theta \int_{\R^2} \!\!\!dq_1 \int_{\R^2} \!\!\!dq_2 \Psi_1(q_1,\theta) \overline{\Psi_2(q_2,\theta)} \frac{j_0(|q_2 - q_1|)}{4\pi^2}.
\end{equation}
Even if $\Hil_\Omega(SE(2))$ is not the Gelfand-Raikov Hilbert space on $SE(2)$ generated by a function of positive type over the group algebra $L^1(SE(2))$ (see e.g. \cite[\S 3.3]{Folland95}), it may be worth noting that a similar construction can be drawn in the following way. Let $\phi$ be the bounded $\mathcal{C}^\infty(SE(2))$ function
$$
\phi(g) \doteq \langle \Pi^\Omega(g) c, c\rangle_{L^2(S^1)} = \frac{1}{(2\pi)^2} j_0(\Omega|q|) \ , \quad g = (q,\theta) \in SE(2)
$$
where $c = \frac{1}{2\pi}$ is a constant function on $S^1$, and let $D \in \Sw'(SE(2))$ be the distribution defined by
$$
\langle D, \overline{\Psi}\rangle_{\Sw'\Sw} = \int_{\R^2} \Psi(q,0) dq \quad \forall \ \Psi \in \Sw(SE(2)) .
$$
Then the $\Hil_\Omega(SE(2))$ scalar product (\ref{eq:spr}) can be written as
$$
\langle \Psi_1, \Psi_2 \rangle_{\Hil_\Omega(SE(2))} = \langle D , \overline{(\Psi_1^* \ast \Psi_2)\phi}\rangle_{\Sw' \Sw}
$$
for all $\Psi_1, \Psi_2 \in \Sw(SE(2))$, where $\Psi^*(g) \doteq \overline{\Psi(g^{-1})}$ and $\ast$ is the group convolution.


\begin{proof}[Proof of Theorem \ref{theo:main}, i)]\ \\
We prove that for any $u_0 \in L^2(S^1)$ the family $\left\{\Pi^\Omega(g)u_0\right\}_{g \in SE(2)}$ satisfies the following Parseval-type identity between $L^2(S^1)$ and $\Hil_\Omega(SE(2))$
\begin{equation}\label{eq:Parseval}
\|A^\Omega\Phi\|_{\Hil_\Omega(SE(2))} = \|u_0\|_{L^2(S^1)} \|\Phi\|_{L^2(S^1)} \quad \forall \ \Phi \in L^2(S^1)\ .
\end{equation}
In order to see this, it suffices to note that
\begin{align*}
\langle A^\Omega\Phi(\cdot,\theta) , \psi &\rangle_{\Sw'(\R^2)\Sw(\R^2)}\\
& = \int_{\R^2} dq \int_{S^1} d\varphi \, \Phi(\varphi) e^{i\Omega(q_1\cos\varphi + q_2\sin\varphi)} \overline{u_0(\varphi - \theta)} \overline{\psi(q)}\\
& = \int_{S^1} d\varphi \, \Phi(\varphi) \overline{u_0(\varphi - \theta)\hat{\psi}(\Omega\cos\varphi,\Omega\sin\varphi)}\\
& = \langle \hat{T}^\Omega_{\Phi\,\overline{u_0(\cdot\, - \theta)}} , \F\psi \rangle_{\Sw'(\hat{\R}^2)\Sw(\hat{\R}^2)} .
\end{align*}
From here we immediately deduce
\begin{align*}
\int_{S^1} d\theta \|A^\Omega\Phi(\cdot,\theta)\|^2_{\Hil_\Omega} & = \int_{S^1} d\theta \int_{S^1} d\varphi |u_0(\varphi - \theta) \Phi(\varphi)|^2\\
& = \|u_0\|^2_{L^2(S^1)} \|\Phi\|^2_{L^2(S^1)} . \qedhere
\end{align*}
\end{proof}

For any nonzero $u_0 \in L^2(S^1)$, the function
$$
SE(2) \times SE(2) \ni (g,h) \mapsto K^\Omega(g,h) = \langle \Pi^\Omega(h) u_0, \Pi^\Omega(g) u_0 \rangle_{L^2(S^1)}
$$
is a positive definite kernel, or Mercer kernel, on $SE(2)$, in the sense that for any sequence $\{g_i\}_{i = 1 \dots n} \subset SE(2)$ and any sequence $\{a_i\}_{i = 1 \dots n} \subset \C$ it holds
$$
\sum_{i=1}^n \sum_{j=1}^n a_i \overline{a_j} K(g_j,g_i) = \left\| \sum_{i=1}^n a_i \Pi^\Omega(g_i) u_0 \right\|_{L^2(S^1)}^2 \geq 0 .
$$

$K^\Omega$ defines then a unique reproducing kernel Hilbert spaces (see e.g. \cite{Aronszajn} or \cite[Chapt. 4]{SC}), that is the target space of the operator $A^\Omega$, which reads
$$
\Hil_{K^\Omega} = \{F \in \C^{SE(2)} \,|\, \exists v \in L^2(S^1) \ \textnormal{such that} \ F = \langle \Pi^\Omega(\cdot)u_0, v\rangle_{L^2(S^1)}\}
$$
and is endowed with the norm 
$$
\|F\|_{\Hil_{K^\Omega}} = \inf \{\|v\|_{L^2(S^1)} : v \in L^2(S^1), F = \langle \Pi^\Omega(\cdot)u_0, v\rangle_{L^2(S^1)}\} .
$$

By (\ref{eq:Parseval}) we have that $\Hil_{K^\Omega}$ is a closed Hilbert subspace of $\Hil_\Omega(SE(2))$, whose norm provides an explicit expression for $\| \cdot \|_{\Hil_{K^\Omega}}$. Indeed, for $g \in SE(2)$ call $u_g(\varphi) \doteq \Pi^\Omega(g)u_0(\varphi)$ and $K^\Omega_g(h) \doteq K^\Omega(g,h)$. Then
$$
K^\Omega_g(h) = A^\Omega u_g (h)
$$
and, if $u_0$ has unit norm, by polarization of (\ref{eq:Parseval}) we have
$$
\langle K^\Omega_g , A^\Omega \Phi \rangle_{\Hil_\Omega(SE(2))} = \langle u_g , \Phi\rangle_{L^2(S^1)} = A^\Omega \Phi(g)
$$
for all $\Phi$ in $L^2(S^1)$. By uniqueness of the reproducing kernel Hilbert space associated to a Mercer kernel one then has that $\| F \|_{\Hil_{K^\Omega}} = \| F \|_{\Hil_{\Omega}(SE(2)}$ for all $F \in \Hil_{K^\Omega}$. The abstract setting for this approach was settled in \cite{Saitoh}, where the author makes use of the hypothesis that $u_0$ is a cyclic vector, which is always the case for irreducible representations.

As for the standard square integrable cases, the isometry associated to (\ref{eq:Parseval}) allows also a reconstruction formula. In this case, however, it differs from (\ref{eq:reproducing0}) in that it is not a standard linear superposition. What can be obtained is indeed the reconstruction of the linear functional associated to the analyzed vector.

\begin{prop}
Let $\Phi\in L^2(S^1)$ and $A^\Omega \Phi$ as in (\ref{eq:analysis}), with $u_0$ of unit norm in $L^2(S^1)$. Then the following identity holds in distributional sense
$$
\F_{SE(2)} \overline{A^\Omega \Phi} (\cdot)u_0 = \hat T^\Omega_\Phi
$$
which means that for all $\psi \in \Sw(\hat{\R}^2)$
\begin{displaymath}
\langle \F_{SE(2)} \overline{A^\Omega \Phi} (\cdot)u_0, \psi\rangle_{\Sw'\Sw} = \langle\Phi,\psi^\Omega\rangle_{L^2(S^1)}\ .
\end{displaymath}
\end{prop}

\begin{proof}
By computations analogous to those used for (\ref{eq:Parseval}), we have that the distributional Fourier transform of $A^\Omega \Phi(q,\theta)$ with respect to spatial variables reads
\begin{equation}\label{eq:FF}
\F A^\Omega \Phi(\cdot ,\theta) = \overline{u_0(\varphi - \theta)}\,\Phi(\varphi)\,\hat T_1^\Omega \quad \textrm{a.e.} \ \varphi \in S^1 .
\end{equation}
Now by the normalization of $u_0$ and since $\Phi\,\hat T_1^\Omega = \hat T_\Phi^\Omega$, then
\begin{align*}
\hat T_\Phi^\Omega
& = \int_0^{2\pi} d\theta\, u_0(\varphi - \theta) \F A^\Omega\Phi(k,\theta)\\
& = \int_0^{2\pi} d\theta \int_{\R^2} \frac{dq}{2\pi} e^{-i\kappa(q_1\cos\varphi + q_2\sin\varphi)}u_0(\varphi - \theta) A^\Omega\Phi(q,\theta)\ . \qedhere
\end{align*}

\end{proof}

\section{The uncertainty principle and CR functions}\label{sec:CR}

In this section we provide a characterization of the target space of the continuous wavelet transform in the case of a mother wavelet that is a minimum of the uncertainty principle for the $SE(2)$ group. In this case the reproducing kernel Hilbert space of $A^\Omega$ is the subspace of $\Hil_\Omega(SE(2))$ of functions which are complex regular with respect to the natural $CR$ structure of $SE(2)$. This statement is the analogous of the well-known result of Bargmann \cite{BargmannII}, since it is stated in terms of an integrability and complex regularity condition, and for this reason we will denote the resulting analysis the $SE(2)$-Bargmann transform.

\subsection{CR structures in $\R^2 \times S^1$}

The manifold $\R^2 \times S^1$ can be considered as the real submanifold of $\C^2$ given explicitly in (\ref{eq:realsubm}), with the induced $CR$ structure. This is a linear operator $J$ that acts on the Lie algebra generators (\ref{generators}) as
\begin{equation}\label{eq:almostcomplex}
J(X_2) = \lambda X_1 \ , \quad J(X_1) = - \frac{1}{\lambda} X_2
\end{equation}
where the constant $\lambda$ arise from a plain rescaling of the spatial variables $q$.
Indeed, the ordinary complex structure of $\C^2$ with respect to $z_j = p_j + i\frac{q_j}{\lambda}$
\begin{displaymath}
J(\partial_{p_j}) = \lambda\partial_{q_j} \ , \quad J(\partial_{q_j}) = - \frac{1}{\lambda}\partial_{p_j}
\end{displaymath}
when restricted to (\ref{eq:realsubm}) reduces to (\ref{eq:almostcomplex}).

The antiholomorphic, or $CR$, vector field for the $CR$ structure (\ref{eq:almostcomplex}) is $Z = X_2 + i \lambda X_1$, since
\begin{displaymath}
J(Z) = - i Z
\end{displaymath}
so we are lead to the following definition.
\begin{defi}\label{def:CRl}
We say that $F : \R^2_q \times S^1_\theta \rightarrow \C$ is $CR^\lambda(SE(2))$ if
\begin{equation}\label{eq:CRlambda}
\left(X_2 + i \lambda X_1\right) F = 0
\end{equation}
where $X_1$ and $X_2$ are the left invariant differential operators given by (\ref{generators}).\\
\end{defi}

\subsection{The uncertainty principle}

In \cite[Th. 2.4]{FS} it is stated a general form of the Uncertainty Principle for connected Lie groups. In our situation the statement reduces to

\begin{theo}[Uncertainty principle]
Let $X_1$ and $X_2$ be the left invariant differential operators given by (\ref{generators}) and let $d\Pi^\Omega$ the Lie algebra representation corresponding to (\ref{eq:S1representation}). Then
$$
\|d\Pi^\Omega(X_1)u\|_{L^2(S^1)} \, \|d\Pi^\Omega(X_2)u\|_{L^2(S^1)} \geq \frac12 |\langle d\Pi^\Omega([X_1,X_2]) u, u\rangle_{L^2(S^1)}|
$$
for all $u \in L^2(S^1)$. The inequality becomes an equality if $u$ satisfies
$$
\left(d\Pi^\Omega(X_1) - i \lambda d\Pi^\Omega(X_2) \right) u = 0
$$
for some $\lambda \in \R$, and we call a solution to this equation a minimal uncertainty state.
\end{theo}

\begin{prop}\label{prop:minunc}
The equation for minimal uncertainty states of the $SE(2)$ group with respect to representation (\ref{eq:S1representation}) reads explicitly
\begin{equation}\label{eq:minunc}
\left(\frac{d}{d\varphi} + \lambda \Omega \sin\varphi \right) u (\varphi) = 0
\end{equation}
and we denote by $u^{\lambda,\Omega}$ its $L^2(S^1)$ normalized solution
\begin{displaymath}
u^{\lambda,\Omega} (\varphi) = \frac{1}{\sqrt{j_0(-2 i \lambda \Omega)}} e^{\lambda\Omega\cos\varphi}\ .
\end{displaymath}
\end{prop}

We note in passing that the uncertainty principle can be stated in slightly more general terms, including mean values of the noncommuting operators. In that case, the equation for minimal uncertainty states becomes an eigenvalue equation. With respect to the present situation, the physical interpretation of minimal uncertainty states (\ref{eq:minunc}) is that of having zero average angular momentum \cite{CN}.

\begin{proof}[Proof of Proposition \ref{prop:minunc}]
It suffices to show that $d\Pi^\Omega(X_1)$ is the multiplicative operator $i \Omega \sin\varphi$ and  $d\Pi^\Omega(X_2)$ is the derivation $\displaystyle{\frac{d}{d\varphi}}$, that can be easily checked by direct computation.
\end{proof}

\begin{defi}
We denote $SE(2)$-Bargmann transform the continuous wavelet transform with respect to the family $\left\{\Pi^{\Omega}(q,\theta)u^{\lambda,\Omega}\right\}_{(q,\theta) \in SE(2)}$
\begin{equation}\label{eq:SE2Barg}
\B_{\Omega}^{\lambda}\Phi(q,\theta) \doteq \langle \Pi^{\Omega}(q,\theta) u^{\lambda,\Omega}, \Phi\rangle_{L^2(S^1)} \ , \quad \Phi \in L^2(S^1)\ .
\end{equation}
\end{defi}

The following general fact provides the relation between uncertainty principle and $CR$ functions, and actually motivates the previous definition.

\begin{lem}\label{lem:CR}
Let $\G$ be a Lie group and $\pi$ be a unitary representation of $\G$ on a Hilbert space $\Hil$. If $X$ is a left invariant vector field and $u_0$ is a vector of $\Hil$ in the domain of $d\pi(X)$ such that
$$
d\pi(X) u_0 = 0
$$
then for all $\Phi \in \Hil$
$$
X \langle \pi(g) u_0, \Phi\rangle_{\Hil} = 0 .
$$
\end{lem}
\begin{proof}
By definition of left invariant vector field, if $F$ is a smooth function on $\G$ then (see e.g. \cite[Th. 3.2.3]{Cohn})
$$
X F (g) = \frac{d}{dt}\Big|_{t = 0} F(g \cdot \exp(t X))
$$
so, since $\pi$ is a homomorphism
$$
X \langle \pi(g) u_0, \Phi\rangle_{\Hil} = \langle \pi(g) \frac{d}{dt}\Big|_{t = 0} \pi(\exp(t X)) u_0, \Phi\rangle_{\Hil}
$$
and by definition of algebra representation $\displaystyle{\frac{d}{dt}}\Big|_{t = 0} \pi(\exp(t X)) = d\pi(X)$.
\end{proof}

By Lemma (\ref{lem:CR}), since $(d\Pi^\Omega(X_2) - i \lambda d\Pi^\Omega(X_1)) u^{\lambda,\Omega} = 0$, then $\B_{\Omega}^{\lambda}\Phi$ satisfies equation (\ref{eq:CRlambda}). We then have the following corollary.
\begin{cor}\label{cor:CR}
Any function on $\R^2 \times S^1$ obtained as an $SE(2)$-Bargmann transform (\ref{eq:SE2Barg}) is $CR^\lambda(SE(2))$ in the sense of Definition \ref{def:CRl}.
\end{cor}


\subsection{The target space}

These notions allow us to define the Hilbert space of surjectivity for the $SE(2)$-Bargmann transform.

\begin{defi}
We define the $SE(2)$-Bargmann space as
\begin{equation}\label{eq:eFock}
\PF{\lambda}{\Omega} = \Hil_\Omega(SE(2)) \cap CR^\lambda(SE(2))\ .
\end{equation}
\end{defi}
\noindent
On the basis of this definition we can conclude the proof of our main result.

\begin{proof}[Proof of Theorem \ref{theo:main}, ii)]\ \\
We shall prove here that $\B_{\Omega}^{\lambda}: L^2(S^1) \rightarrow \PF{\lambda}{\Omega}$ is an isometric surjection.
Recall that by Corollary \ref{cor:CR} we know that the $SE(2)$-Bargmann transforms are $CR^\lambda(SE(2))$ functions, while Theorem \ref{theo:main}, i) and the normalization of $u^{\lambda,\Omega}$ ensure that this is an isometry.

To prove surjectivity, consider a function $F(q,\theta) \in \PF{\lambda}{\Omega}$. Since it satisfies (\ref{eq:CRlambda}), its Fourier transform with respect to the $q$ variables is such that
\begin{equation}\label{CRf}
\left(\partial_\theta + \lambda \kappa \sin(\theta - \varphi)\right) \F F(\kappa\cos\varphi,\kappa\sin\varphi,\theta) = 0\ .
\end{equation}
Hence
\begin{displaymath}
\F F(\kappa\cos\varphi,\kappa\sin\varphi,\theta) = c\,e^{\lambda \kappa \cos(\theta - \varphi)} \Phi(\varphi) g(\kappa)
\end{displaymath}
for some $\Phi(\varphi), g(\kappa)$. But since $F$ is in $\mathcal{H}_\Omega(SE(2))$, then $\Phi \in L^2(S^1)$ and $g(\kappa) = c\, \delta(\kappa - \Omega)$, so any $F(q,\theta) \in \PF{\lambda}{\Omega}$ is determined by a $\Phi \in L^2(S^1)$.
\end{proof}

\subsection{Relation with the Bargmann transform}

In this section we will interpret the $SE(2)$-Bargmann transform as the restriction of the ordinary Bargmann transform on $L^2(\R^2)$ to the real submanifold (\ref{eq:realsubm}).

If we denote Euclidean translations with $\tau$ and modulations with $\mu$, i.e.
\begin{displaymath}
\tau(q)\mu(p)f(x) = e^{ip\cdot (x - q)} f(x - q)
\end{displaymath}
then the Bargmann transform \cite{BargmannII} is
\begin{equation}\label{eq:H2Barg}
\B^{\Heis^2}_\sigma f(q,p) = e^{\frac{\sigma^2 |p|^2}{2}} \langle \tau(q)\mu(p)g_0 , f \rangle_{L^2(\R^n)}
\end{equation}
where $g_0$ is an $L^2(\R^2)$ normalized Gaussian of width $\sigma$
$$
g_0(x) = \frac{1}{\sigma\sqrt{\pi}} e^{-\frac{|x|^2}{2\sigma^2}} .
$$

For notation purposes we recall the following well known result.
\begin{prop}\label{BH2}
The Bargmann transform is an isometric surjection from $L^2(\R^2)$ onto the Fock-Bargmann space
\begin{displaymath}
\Fock_\sigma = L^2(\R^2_q \times \R^2_p, e^{-\frac{\sigma^2 |p|^2}{2}}dqdp) \cap Hol(\C^2_{z_1,z_2})
\end{displaymath}
with holomorphy with respect to the complex structure $z_j = p_j + i \frac{q_j}{\sigma^2}$
\begin{equation}\label{analytic}
\left(\partial{p_j} + i\sigma^2\partial{q_j}\right) \B^{\Heis^2}_\sigma f = 0 \quad , \quad j=1,2 \ .
\end{equation}
\end{prop}

In order to study the relation between (\ref{eq:SE2Barg}) and (\ref{eq:H2Barg}) we note that the $CR$ structure (\ref{eq:almostcomplex}) is inherited by the same complex structure of $\C^2$ used for (\ref{analytic}) with $\lambda = \sigma^2$. In particular this implies that (see e.g. \cite{BER}) if we restrict a function satisfying (\ref{analytic}) to the real submanifold (\ref{eq:realsubm}) we obtain exactly what we have called a $CR^{\sigma^2}(SE(2))$ function. Since the $p$ variables are cotangent variables, i.e. they belong to the frequency domain, the deep relation between the two Bargmann transforms is exploited by means of a restriction of the function to be transformed to circles in the frequency domain, i.e. by means of the quasi-projector $\Proj_\Omega$.

The main theorem of this section is the following.
\begin{theo}\label{theo:Bargmann}
Let $T_\Omega$ be a distribution in $\Hil_\Omega$ and $\hat{T}^\Omega_\Phi \in \Hil^\Omega$ be its distributional Fourier transform associated to $\Phi \in L^2(S^1)$. Then considering the extension of $\B^{\Heis^2}_\sigma$ to tempered distributions as in \cite{BargmannII}
$$
\B^{\Heis^2}_\sigma T_\Omega(q,p) = e^{-\frac{\sigma^2 \Omega^2}{2}} \frac{\sqrt{j_0(-2 i \sigma^2 |p| \Omega)}}{\sigma\sqrt{\pi}}  \B_{\Omega}^{\sigma^2 |p|}\Phi(q,\theta)
$$
where $p = |p|(\cos\theta,\sin\theta)$.
\end{theo}
\begin{proof}
We first show that if $f$ in $\Sw(\R^2)$, then
\begin{equation}\label{dum2}
\B^{\Heis^2}\Proj_\Omega f(q,p) = e^{-\frac{\sigma^2 \Omega^2}{2}} \frac{\sqrt{j_0(-2 i \sigma^2 |p| \Omega)}}{\sigma\sqrt{\pi}}  \B_{\Omega}^{\sigma^2 |p|}\hat{f}^{\Omega}(q,\theta)\ .
\end{equation}
This can be seen by Fourier duality since
\begin{eqnarray*}
\B^{\Heis^2}\Proj_\Omega f(q,p) & = & e^{\frac{\sigma^2 |p|^2}{2}} \langle \tau(q)\mu(p)g_0 , \Proj_\Omega f \rangle_{\Sw'(\R^2),\Sw(\R^2)} \\
\hspace{-100pt} & \hspace{-100pt} = & \hspace{-50pt} e^{\frac{\sigma^2 |p|^2}{2}} \langle \F\tau(q)\mu(p)g_0 , \Proj^\Omega\F f \rangle_{\Sw'(\hat{\R}^2),\Sw(\hat{\R}^2)}\\
\hspace{-100pt} & \hspace{-100pt} = & \hspace{-50pt} \frac{1}{\sigma\sqrt{\pi}} \displaystyle{e^{\frac{\sigma^2 |p|^2}{2}} \int_{\R^2} e^{ik\cdot q} e^{-\frac{\sigma^2|k-p|^2}{2}} \Proj^\Omega\hat{f}(k)} dk\\
\hspace{-100pt} & \hspace{-100pt} = & \hspace{-50pt} \frac{1}{\sigma\sqrt{\pi}} \displaystyle{\int_0^{+\infty} d\kappa \kappa e^{-\frac{\sigma^2 \kappa^2}{2}}\int_{S^1} d\varphi e^{i\kappa(q_1\cos\varphi + q_2\sin\varphi)}e^{\sigma^2 |p| \kappa\cos(\varphi - \theta)} \Proj^\Omega\hat{f}(k)}\\
\hspace{-100pt} & \hspace{-100pt} = & \hspace{-50pt} \frac{\sqrt{j_0(-2 i \sigma^2 |p| \Omega)}}{\sigma\sqrt{\pi}} e^{-\frac{\sigma^2 \Omega^2}{2}} \langle \Pi^{\Omega}(q,\theta)u^{\sigma^2 |p|,\Omega} , \hat{f}^\Omega \rangle_{L^2(S^1)}
\end{eqnarray*}
By the results of Section \ref{sec:HilOmega} we can take $f$ in $\Hil_\Omega$, thus extending (\ref{dum2}) to the whole $\Hil_\Omega$.
\end{proof}

\end{document}